\documentclass[11pt]{article}

\usepackage{amsthm} 

\newtheorem{theorem}{Theorem}[section]
\newtheorem{proposition}[theorem]{Proposition}
\newtheorem{lemma}[theorem]{Lemma}

\begin{document}

\title{The Tur\'an number of sparse spanning graphs}

\author{
Noga Alon
\thanks{Sackler School of Mathematics
and Blavatnik School of Computer Science,
Tel Aviv University,
Tel Aviv 69978, Israel. E-mail: nogaa@tau.ac.il.
Research supported in part by an ERC advanced grant, by
a USA-Israeli BSF grant, and by the Israeli I-Core program.
} \and
Raphael Yuster
\thanks{Department of Mathematics, University of Haifa, Haifa
31905, Israel. E--mail: raphy@research.haifa.ac.il}
}

\date{}

\maketitle

\setcounter{page}{1}

\begin{abstract}
For a graph $H$, the {\em extremal number} $ex(n,H)$ is the maximum number
of edges in a graph of order $n$ not containing a subgraph isomorphic
to $H$.  Let $\delta(H)>0$ and $\Delta(H)$ denote the minimum degree
and maximum degree of $H$, respectively.  We prove that for all $n$
sufficiently large, if $H$ is any graph of order $n$ with $\Delta(H) \le
\sqrt{n}/200$, then $ex(n,H)={{n-1} \choose 2}+\delta(H)-1$.  The condition
on the maximum degree is tight up to a constant factor.  This generalizes
a classical result of Ore for the case $H=C_n$, and resolves, in a strong
form, a conjecture of Glebov, Person, and Weps for the case of graphs.
A counter-example to their more general conjecture concerning the extremal
number of bounded degree spanning hypergraphs is also given.
\end{abstract}

\section{Introduction}

Ore \cite{ore-1961} proved that a non-Hamiltonian graph of order $n$ has at most ${{n-1} \choose 2}+1$ edges.
Ore's theorem can be expressed as a spanning Tur\'an-type result.
For a graph $H$, the {\em extremal number} $ex(n,H)$ is the maximum number of edges in a graph of order $n$ not containing a subgraph isomorphic to $H$.
Hence, Ore's theorem is that $ex(n,C_n)={{n-1} \choose 2}+1$, where $C_n$ is the cycle of order $n$.

Recently, Ore's theorem has been generalized to the setting of Hamilton cycles in hypergraphs.
For integers $n > k > \ell \ge 0$ where $(k-\ell) | n$, the $(k,\ell)$-tight cycle of order $n$, denoted by $C_n^{(k,l)}$, is the $k$-uniform hypergraph on vertex set $[n]$ and
edges $\{i(k-\ell) + 1, i(k-\ell) + 2, \ldots , i(k-\ell) + k\}$ for $0 \le i < \frac{n}{k-\ell}$ (addition modulo $n$).
In particular, $C_n = C_n^{(2,1)}$.
Extending and generalizing earlier results of Katona and Kierstead \cite{KK-1999} and of Tuza \cite{tuza-2006},
the extremal number $ex(n,C_n^{(k,l)})$ was determined by Glebov, Person and Weps in \cite{GPW-2012} for all $k$ and $\ell$, assuming $n$ is sufficiently large
and $(k-\ell) | n$. Their result extends Ore's theorem to the hypergraph setting.
The extremal number $ex(n,C_n^{(k,l)})$ is of the form ${{n-1} \choose k} + ex(n-1,P)$ where
$P$ is the $(k-1,\ell-1)$-tight path (defined in the obvious way)
with $\lfloor k/(k-\ell) \rfloor(k-\ell)+\ell-1$ vertices.

It is natural to try to extend Ore's result to spanning structures other than just Hamilton cycles, in both the graph and hypergraph settings.
Suppose that $H$ is a $k$-uniform hypergraph of order $n$ and with, say, bounded maximum degree. 
It is natural to suspect that $ex(n,H) \le {{n-1} \choose k} + ex(n-1,S)$ where $S$ is some set of $(k-1)$-uniform hypergraphs that depend on the neighborhood
structure of $H$. A conjecture raised in \cite{GPW-2012} asserts that it suffices to take $S$ to be the set of {\em links} of $H$.
The link of a vertex $v$ in a $k$-uniform hypergraph $H=(V,E)$
is the $(k-1)$-uniform hypergraph
$H(v) = (V\setminus\{v\},E_v)$ with $\{x_1, \ldots, x_{k-1}\}\in E_v$ iff $\{v,x_1,\ldots, x_{k-1}\} \in E$.
For example, the links of a vertex of $C_n^{(3,1)}$ are either the graph with a single edge or the graph with two independent edges.
The link of a vertex of $C_n^{(3,2)}$ is the graph $P_4$ (the path on four vertices).
Conjecture 9 in \cite{GPW-2012} states that 
$ex(n,H) \le {{n-1} \choose k} + ex(n-1,{\cal L})$ 
where ${\cal L}$ is the set of links of $H$ and $ex(n-1,{\cal L})$
denotes the maximum number of edges in a $(k-1)$-uniform
hypergraph on $n$ vertices that contains none of the links 
of $H$.
Observe that this conjecture holds for both Ore's Theorem and its aforementioned generalization to Hamilton cycles in hypergraphs.
In fact, it holds with equality in these cases.

In the graph-theoretic case, the link of a vertex is just a set of singletons whose cardinality is the degree of the vertex.
In this case, the aforementioned conjecture states that if $H$ is a 
graph of order $n$ with minimum degree $\delta>0$, 
and bounded maximum degree, then
$ex(n,H)={{n-1} \choose 2}+\delta-1$ assuming $n$ is sufficiently large.
(Clearly, we cannot expect to have a sharp inequality, as an $(n-1)$-clique together with an additional vertex that
is connected only to $\delta-1$ vertices of the clique does not contain $H$ as a spanning subgraph.)
Indeed, our main result is a proof of this conjecture in a strong sense. We do not require the maximum degree of $H$ to be bounded
independently of $n$.
\begin{theorem}
\label{t:main}
For all $n$ sufficiently large, if $H$ is any graph of order $n$ 
with no isolated vertices and 
$\Delta(H) \le \sqrt{n}/200$, then $ex(n,H)={{n-1} \choose 2}+\delta(H)-1$.
\end{theorem}
In our proof we make no attempt to minimize the value of $n$ starting from which the theorem holds, although it can be worked out from the proof to be less than
$100000$. Also, the constant $200$ in the bound for the maximum degree is not optimal and can be somewhat improved.
However, more importantly, the following construction shows that it {\em cannot} be improved to less than $\sqrt{2}$, 
and hence the $O(\sqrt n)$ bound on the maximum  degree
is optimal up to a constant factor. Consider the graph $H$ with $n=k(k+6)/2+1$ vertices, consisting of $k$ vertex-disjoint cliques of size  $(n-1)/k$ each,
and an additional vertex connected to some $\delta \le (n-1)/k-1$ vertices of the cliques.
Clearly, $\Delta(H)=(n-1)/k$ and $\delta(H)=\delta$. Observe, however, that $H$ does not have an independent set of size $k+2$.
Hence, if $G$ is the $n$-vertex graph obtained from $K_n$ by removing a $K_{k+2}$, then $H$ is not a spanning subgraph of $G$.
However, $G$ has ${n \choose 2}-{{k+2} \choose 2}$ edges, which is easily checked to be more than ${{n-1} \choose 2}+\delta(H)-1$.

Our next result shows that the conjecture of Glebov et al. is 
false already for $3$-uniform hypergraphs. To simplify the
presentation we describe one example, the same 
proof can provide many others.
\begin{proposition}
\label{p:main}
Let $s$ be a large integer, define $n=1+5s$
and let
$V=V_1 \cup V_2 \cup \cdots \cup V_s \cup \{x\}$ be a set of
$n$ vertices, where each $V_i$ is a set of $5$ vertices, the sets
$V_i$ are pairwise disjoint, and $x$ is an additional vertex. Let
$H$ be 
the $3$ uniform hypergraph on the set of vertices $V$, where each
$V_i$ forms a complete $3$-uniform hypergraph on $5$ vertices, and
$x$ is contained in a unique edge $\{x,u,v\}$ with $u,v \in V_1$.
Let ${\cal L}=\{H(v): ~v \in V\}$  be the set of all
links of $H$. Then $ex(n-1,{\cal L}) =0$
but 
$$
ex(n,H) \geq {{n-2} \choose 3} +  \frac{4}{3} {{n-2} \choose
2}={{n-1} \choose 3}+\frac{1}{3}{{n-2} \choose 2}
>{{n-1} \choose 3}+ex(n-1, {\cal L}).
$$
\end{proposition}
The reason that $H$ forms a counter-example is that the link of every vertex of $H$ besides one is a complete graph on $4$ vertices, and hence any hypergraph containing two vertices whose links are $3$-colorable cannot contain a spanning copy of $H$. Since a $3$-colorable graph on a set of vertices can contain more than half of all potential edges on this set, $ex(n,H)$ is larger than the number of edges of a complete $3$-uniform hypergraph on $n-1$ vertices.

The rest of this short 
paper is organized as follows. Theorem \ref{t:main} is proved in 
Section 2, Section 3 contains the proof of Proposition \ref{p:main},
and Section 4 contains some concluding remarks.
Throughout the paper we use the standard graph-theoretic terminology 
and notations following \cite{bollobas-1978}.

\section{Proof of the main result}

We say that two graphs $G$ and $H$ of the same order {\em pack}, if
$H$ is a spanning subgraph of the complement of $G$.
Let $H=(W,F)$ be a given graph with $n$ vertices and with $\Delta (H) \le \sqrt{n}/200$.
Let $G=(V,E)$ be any graph with $n$ vertices and $n-\delta-1$ edges, where $\delta=\delta(H)$.
It suffices to prove that $G$ and $H$ pack. Equivalently, we construct a bijection $f:V \rightarrow W$
such that for all $(u,v) \in E$, $(f(u),f(v)) \notin F$. Throughout the proof we assume
that $n$ is larger than some absolute constant.

Before describing $f$, we require some notation.
Let $d(v)$ denote the degree of a vertex $v$ in $G$,
let $N(v)$ denote the set of neighbors of $v$ and, for a subset of vertices $W$, let $N[W]=W \cup (\cup_{w \in W} N(w))$.
Let $V=\{v_1,\ldots,v_n\}$ where $d(v_i) \ge d(v_{i+1})$ for $i=1,\ldots,n-1$.
Let $S_1 \subset V \setminus N[v_1]$ be an independent set of vertices each with degree smaller than $2 \sqrt n$ 
and with maximum possible cardinality under this restriction.
In the next lemma we show that $|S_1| \ge \delta$ and thus let $B_1 \subset S_1$ denote an arbitrary subset with $|B_1|=\delta$.
For $i=2,\ldots,n$, let $S_i \subset V \setminus (N[v_i] \cup N[B_1])$ be an independent set of vertices
but with the additional requirement that each $u \in S_i$ has $d(u) \le 10$. Furthermore, we require that $S_i$ has maximum possible cardinality under these
restrictions.
\begin{lemma}
\label{l:1}
The following bounds hold:
\begin{enumerate}
\item
$d(v_1) \le n-\delta-1$, $d(v_2) \le n/2$, and $d(v_i) < 2n/i$.
\item
$|S_1| \ge \delta$ and $|S_i| \ge n/18$ for $i=2,\ldots,n$.
\end{enumerate}
\end{lemma}
\begin{proof}
Trivially, $d(v_1) = \Delta(G) \le |E|=n-\delta-1$.
Since $d(v_1)+d(v_2) \le |E|+1$, we have that $d(v_2) \le (|E|+1)/2 \le n/2$.
Since $\sum_{i=1}^n d(v_i) =2n-2\delta-2 < 2n$, we have that $d(v_i) < 2n/i$.

The subgraph of $G$ consisting of the non-neighbors of $v_1$ 
has $n-d(v_1)-1$ vertices and
at most $n-\delta-1-d(v_1)$ edges and therefore has an independent set of size at least $(n-d(v_1)-1)/3$.
If, say, $n-d(v_1)-1 \geq 6 \delta$, at least  $(n-d(v_1)-1)/6 \geq \delta$
of its vertices have degree at most $6 < 2 \sqrt n$. Thus,
in this case $|S_1| \geq \delta$.
Otherwise, any non-neighbor of $v_1$ has degree at most $6 \delta
< 2 \sqrt n$, and as the subgraph induced on them has
$n-d(v_1)-1$ vertices and at most $n-\delta-1-d(v_1)$ edges
it has at least $\delta$ components. Selecting one vertex from each
component shows that $|S_1| \ge \delta$ in this case as well.

By a similar reasoning, the subgraph of $G$ consisting of 
the non-neighbors of $v_i$ does not have more edges than vertices,
and hence has an independent set of size at least $1/3$ of its cardinality. As for $i \ge 2$, its cardinality is at least
$n-d(v_i)-1 \ge n-d(v_2)-1 \ge n-n/2-1=n/2-1$,
it has an independent set of size at least $n/6-1$.
Let $S'_i$ denote the subset of this independent set consisting only of vertices whose degrees do not exceed $10$.
We claim that $|S'_i| \ge n/15$. Indeed, otherwise there are
at least $(n/6-1)-n/15=n/10-1$ independent vertices of $G$ with degree at least $11$, which contradicts the fact that $G$ has less than $n$ edges.
Next, we remove from $S'_i$ any vertex which belongs to $N[B_1]$. Recall that $|B_1|=\delta$ and that each vertex of $B_1$ has degree at most
$2\sqrt{n}$. Hence, $|N[B_1]| \le 2\sqrt{n}\delta \le 2\sqrt{n}\Delta \le n/100$.  It follows that $S_i= S'_i \setminus N[B_1] \ge n/15-n/100 \ge n/18$
as required.
\end{proof}

For $i=2,\ldots,n$, let $B_i$ be a random subset of $S_i$,
where each vertex of $S_i$ is independently selected to $B_i$ with probability $n^{-1/2}$. Let
\begin{eqnarray*}
C_i & = & (\cup_{j=2}^{i-1} B_j) \cap N(v_i)\;,\\
D_i & = & B_i \setminus (\cup_{j=2}^{i-1} N[B_j])\;.
\end{eqnarray*}
Clearly, $|C_i|$ could be as large as $d(v_i)$, which, in turn could be as large as $2n/i$.
On the other hand, $|D_i|$ could be as small as zero.
We will need, however, to make sure that $|C_i|$ is considerably smaller than $2n/i$, at least for relatively small $i$,
and that $|D_i|$ is rather large, at least for relatively small $i$. This is made precise in the following lemma.
\begin{lemma}
\label{l:2}
For $n$ sufficiently large, all of the following hold with positive probability:
\begin{enumerate}
\item
$|C_i| \le 4\sqrt{n}$ for $i=2,\ldots,n$,
\item
$|D_i| \ge \sqrt{n}/50$ for $i=2,\ldots,\lceil \sqrt{n}/10 \rceil$.
\end{enumerate}
\end{lemma}
\begin{proof}
We prove that each of the two sets of bounds hold with probability higher than $1/2$, and hence both hold with positive probability.

For the first part of the lemma, we only need to consider vertices  $v_i$ with $d(v_i) \ge 4\sqrt{n}$,
as for other vertices the claim clearly holds since $|C_i| \le |N(v_i)|=d(v_i)$.
So, fix some vertex $v_i$ with $d(v_i) \ge 4\sqrt{n}$. We prove that the probability that $|C_i| > 4\sqrt{n}$
is smaller than $1/(2n)$ and hence, by the union bound, this part of the lemma holds with probability greater than $1/2$.
For $u \in N(v_i)$, the probability that $u \in B_j$ is at most $n^{-1/2}$ (it is either $n^{-1/2}$ if $u \in S_j$ or $0$ if $u \notin S_j$).
As the membership of $u$ in $C_i$ is only determined by its membership in $B_2 \cup \cdots \cup B_{i-1}$, we have that
$\Pr[u \in C_i] \le (i-2)n^{-1/2}$. By Lemma \ref{l:1}, $i < 2n/d(v_i)$. Hence,
$$
\Pr[u \in C_i] \le \frac{2\sqrt{n}}{d(v_i)}\;.
$$
Observe that $|C_i|$ is a sum of $d(v_i)$ independent indicator random variables, one for each $u \in N(v_i)$, each variable
having success probability at most $2\sqrt{n}/d(v_i)$. The expectation of $|C_i|$ is therefore at most $2\sqrt{n}$
and by a large deviation inequality of Chernoff (see \cite{AS-2000}, Theorem A.1.11), the probability of $|C_i|$ being larger than
$4\sqrt{n}$ is exponentially small in $\sqrt{n}$. In particular, for $n$ sufficiently large, it is smaller than $1/(2n)$.

For the second part of the lemma, observe that for $u \in S_i$, the probability that $u \in B_i$ is $n^{-1/2}$.
On the other hand, for any $j \ge 2$, the probability that $u \notin N[B_j]$ is at least $1-10n^{-1/2}$.
Indeed, this is true because either $u \in S_j$ in which case $u$ is selected for $B_j$ (and therefore for $N[B_j]$) with probability at most $1/\sqrt{n}$.
Else, since $u \in S_i$ we already know that $d(u) \le 10$. Hence, $u$ has at most $10$ neighbors in $S_j$ so $u \in N[B_j]$ with probability at most
$10/\sqrt{n}$.
Hence, as long as $i \le \lceil \sqrt{n}/10 \rceil$,
$$
\Pr[u \in D_i] \ge n^{-1/2}(1-10n^{-1/2})^{i-2} \ge \frac{1}{e\sqrt{n}}\,.
$$
Observe that $|D_i|$ is a sum of $|S_i|$ independent indicator random variables, each having success probability at least $\frac{1}{e\sqrt{n}}$.
By Lemma \ref{l:1}, $|S_i| \ge n/18$, and therefore the expectation of $|D_i|$ is at least $\sqrt{n}/(18e) > \sqrt{n}/49$.
By a large deviation inequality of Chernoff (see \cite{AS-2000}, Theorem A.1.13), the probability that $|D_i|$ falls below
say, $\sqrt{n}/50$ is exponentially small in $\sqrt{n}$. In particular, for $n$ sufficiently large, it is smaller than $1/(2n)$, so by the union bound,
the second part of the lemma holds with probability greater than $1/2$ for all $i=2,\ldots,\lceil \sqrt{n} \rceil$.
\end{proof}

\noindent
{\bf Completing the proof of Theorem \ref{t:main}:}\,
By Lemma \ref{l:2} we may fix independent sets $B_2,\ldots,B_n$ satisfying all the conditions of Lemma \ref{l:2} with respect to the
cardinalities of the sets $C_i$ and $D_i$.
The construction of the bijection $f$ is done in four stages.
At each point of the construction, some vertices of $V$ are {\em matched} to some vertices of $W$ while the other vertices
of $V$ and $W$ are yet unmatched. Initially, all vertices are unmatched.
We always maintain the {\em packing} property:
for any two matched vertices $u,v \in V$ with $(u,v) \in E$, their corresponding matches $f(u)$ and $f(v)$ satisfy $(f(u),f(v)) \notin F$.
Thus, once all vertices are matched, $f$ is a packing of $G$ and $H$.

\noindent{\em Stage 1.} We match $v_1$ (which, by definition, is a vertex with maximum degree in $G$) with a vertex $w \in W$
having {\em minimum} degree in $H$, and set $f(v_1)=w$. Let $N(w)$ be the set of neighbors of $w$ in $H$.
As $|N(w)|=\delta$ and since, by Lemma \ref{l:1}, $|S_1| \ge \delta$, we may match $B_1$ with $N(w)$ (recall that $B_1 \subset S_1$ is a set of size $\delta$).
Observe that the packing property is maintained since $S_1$ (and therefore $B_1$) is an independent set of non-neighbors of $v_1$.
Note that after stage 1, precisely $\delta+1$ pairs are matched.

\noindent{\em Stage 2.}
Let $k$ be the largest index such that $d(v_k) \ge 20\sqrt{n}$. Observe that by Lemma \ref{l:1}, $0 \le k \le \sqrt{n}/10$.
This stage is done repeatedly for $i=2,\ldots,k$, where at iteration $i$ we match $v_i$ and some subset of vertices of $B_i$
with a corresponding set of vertices of $W$.
Throughout this stage we maintain the following invariants:
\begin{enumerate}
\item
After iteration $i$ which matches $v_i$ with some vertex $f(v_i)$, we also make sure that all neighbors of $f(v_i)$ in $H$ are matched to vertices of $B_i$.
\item
After iteration $i$, any matched vertex of $V$ other than $v_1,\ldots,v_i$ is contained in $\cup_{j=1}^{i} B_j$.
\item
The overall number of matched vertices after iteration $i$ is at most $i(\Delta(H)+1)$.
\end{enumerate}
Observe that Stage 1 guarantees that these invariants hold at the beginning of Stage 2. Indeed, at the end of Stage 1,
all the invariants hold for $i=1$. In particular, recall that precisely $\delta+1$ vertices have been matched at Stage 1.

So, consider the $i$'th iteration of Stage 2, where $v_i$ is 
some yet unmatched vertex with $d(v_i) \ge 20\sqrt{n}$. (Note that
$v_i$ is indeed yet unmatched as the vertices of each $B_i$,
including $B_1$, have degree smaller than $2 \sqrt n$.)
We partition $N(v_i)$ into three parts $N(v_i)=X \cup Y \cup Z$, where $X$ are the matched neighbors $v_j$ with $j < i$,
$Y$ are the other matched neighbors, and $Z$ being the yet unmatched neighbors.
Clearly $|X| < i \le k \le \sqrt{n}/10$. On the other hand, by the second invariant,
$Y \subset \cup_{j=1}^{i-1} B_j$. Thus, $Y \subset C_i \cup B_1$.
From the first property in Lemma \ref{l:2}, together with $|B_1|=\delta$, we obtain that $|Y| \le \delta + 4\sqrt{n} < 5\sqrt{n}$,
and therefore $|X \cup Y| < 6\sqrt{n}$.

Consider the set $T$ of $|X \cup Y|$ matches of $X \cup Y$ in $H$.
Each vertex of $T$ has at most $\Delta(H)$ neighbors in $H$, so altogether, there is a set $Q$ with
$$
|Q| \ge n-|T|\Delta(H) = n-|X \cup Y|\Delta(H) \ge n-6\sqrt{n} \cdot \sqrt{n}/200 = 97n/100
$$
vertices of $H$ that
are non-neighbors of all vertices of $T$. In order to maintain the packing property, we would like to match $v_i$ with some vertex of $Q$.
In order to do this, we must make sure that $Q$ contains at least one vertex that is yet unmatched.
This, in turn, is true because of the third invariant, as the overall number of matched vertices at this point is only
$(i-1)(\Delta(H)+1) < k(\sqrt{n}/200+1) \le n/2000+\sqrt{n}/10$. So, the number of unmatched vertices is much larger than $n-|Q|$ and hence intersects
$Q$. Let, therefore, $f(v_i)$ be one such vertex.

Let $R$ be the set of neighbors of $f(v_i)$ in $H$ that are still not matched. Clearly, $|R| \le \Delta(H) \le \sqrt{n}/200$.
In order to maintain the first invariant, we must match some unmatched independent set of non-neighbors of $v_i$ with $R$.
A valid choice of such vertices which maintains the packing property is obtained by taking $|R|$ vertices of $D_i= B_i \setminus (\cup_{j=2}^{i-1} N[B_j])$,
and this will also show that the second invariant is maintained (recall also that $S_i$, and hence $B_i$, and hence $D_i$, do not contain vertices of $N[B_1]$).
We can, indeed, pick such a subset, as 
the second property in Lemma \ref{l:2} guarantees that
$|D_i| \ge \sqrt{n}/50 \ge |R|$.
Finally, notice that the third invariant is maintained as iteration $i$ only introduced $|R|+1 \le \Delta(H)+1$ newly matched vertices.

\noindent{\em Stage 3.}
At this point we are guaranteed that the unmatched vertices of $G$ have degree less than $20\sqrt{n}$.
Furthermore, by the third invariant of Stage 2,
the total number of unmatched vertices of $G$ is at least $n-(\sqrt{n}/10)(\Delta(H)+1)  > 19n/20$.
As the unmatched vertices of $G$ induce a subgraph with at least $19n/20$ vertices
and less than $n$ edges, they contain an independent set of size at least $n/4$. Let, therefore, $J$ denote a maximum independent set of
unmatched vertices of $G$ and let $K$ be the remaining unmatched vertices of $G$. We have $|J| \ge n/4$.

The third stage consists of matching the vertices of $K$ one by one. Suppose $v \in K$  is still unmatched. As $d(v) \le 20\sqrt{n}$,
the set $X$ of matched neighbors of $v$ satisfies $|X| \le 20\sqrt{n}$. A similar argument to the one in Stage 2 now follows.
Consider the set $T$ of $|X|$ matches of $X$ in $H$.
Each vertex of $T$ has at most $\Delta(H)$ neighbors in $H$, so altogether, there is a set $Q$ of at least
$n-|X|\Delta \ge n- n/10=9n/10$ vertices of $H$ that
are non-neighbors of all vertices of $T$. In order to maintain the packing property, we would like to match $v$ to some vertex of $Q$.
In order to do this, we must make sure that $Q$ contains at least one vertex that is yet unmatched.
This, in turn, is true because the overall number of matched vertices at this point is at most $n-|J| \le 3n/4$.
So, there is a yet unmatched vertex of $Q$.

\noindent{\em Stage 4.}
It remains to match the vertices of $J$ to the remaining unmatched vertices of $H$, denoted by $Q$.
Construct a bipartite graph $P$ whose sides are $J$ and $Q$. Recall that $|J|=|Q| \ge n/4$.
We place an edge from $v \in J$ to $q \in Q$ if matching $v$ to $q$ is {\em allowed}.
By this we mean  that mapping $v$ to $q$ will not violate the
packing property.
At the beginning of Stage 4, as in Stage 3, for each $v \in J$, there are at least $9n/10$ vertices of $H$
that are non-neighbors of all vertices that are matches of matched neighbors of $v$. So, the degree of $v$ in $P$ is at least $9n/10-(n-|J|) > |J|/2$.

On the other hand, consider some $q \in Q$. Let $T$ be the set of matched neighbors of $q$ in $H$, and let $X$ be their matches in $G$.
Notice that $q$ is not allowed to be matched to an unmatched neighbor of some $x \in X$. So, consider some $x \in X$. Clearly, if $d(x) \ge 20\sqrt{n}$ then
all of the neighbors of $x$ in $G$ are already guaranteed to be matched by the first invariant in Stage 2.
So, let $X' \subset X$ be the vertices with degrees smaller than $20\sqrt{n}$.
Hence, $q$ is not allowed to be matched to at most $|X'|20\sqrt{n}$ vertices of $J$, but
$|X'| \le |X|=|T| \le \Delta(H)$ so $q$ cannot be matched to at most $n/10$ vertices of $J$, which is smaller than $|J|/2$.
Thus, the degree of $q$ in $P$ is also larger than $|J|/2$. It now follows by Hall's Theorem that $P$ has a perfect matching, completing the matching $f$.
\qed

\section{A counter-example for hypergraphs}

In this short section we prove Proposition \ref{p:main}. This 
provides a counter-example to the conjecture of \cite{GPW-2012}.
Indeed, the proposition
gives a $3$-graph $H$ on $n$ vertices whose set of links ${\cal L}$
contains a graph with one edge and hence satisfies
$ex(n-1,{\cal L})=0$. On the other hand,  by the proposition,
$ex(n,H) \geq 
{{n-1} \choose 3}+\frac{1}{3}{{n-2} \choose 2} >
{{n-1} \choose 3}.$ 
\bigskip

\noindent
{\it Proof of Proposition \ref{p:main}.}\,
Let $T$ be the following $3$-uniform hypergraph on the set
of vertices $U \cup \{x,y\}$, where $|U|=n-2$ and $x,y \not \in U$.
Let $U=U_1 \cup U_2 \cup U_3$ be a partition of $U$ into $3$
nearly equal disjoint sets (that is, each $U_i$ is of cardinality
either $\lfloor (n-2)/3 \rfloor$ or $\lceil (n-2)/3 \rceil$.)
The edges of $T$ consist of all $3$-subsets of $U$, as well as
all edges $\{x,u_i,u_j\}$ and $\{y,u_i,u_j\}$ with
$1 \leq i <j \leq 3$, $u_i \in U_i$ and $u_j \in U_j$. 

Note that $T$ does not contain a copy of $H$, as it has $2$
vertices ($x$ and $y$) whose links are $3$-colorable, and thus none
of them lies in a copy of a complete $3$-graph on $5$ vertices.
The desired result follows, as $T$ has  at least
${{n-2} \choose 3} +  \frac{4}{3} {{n-2} \choose 2}$ edges.
\hfill $\qed$
\bigskip

\noindent
{\it Remark.}\, The above reasoning can clearly provide many
additional counter-examples. Indeed, any bounded degree 
$3$-graph in which the links of all vertices but one
are of chromatic number at least $4$, and the last link is
of chromatic number $2$, is a counter-example. There are additional
variants that provide more examples, but for all the ones we know,
the reason that the extremal number is large is local, that is,
one can construct a nearly complete hypergraph $T$ containing 
a set of some $f$ vertices that cannot serve as any set
of $f$ vertices of $H$, where $f$ is bounded by a function of
$\Delta(H)$. It will be interesting to decide if for any
bounded degree hypergraph $H$ on $n$ vertices with no isolated
ones,
the correct value of $ex(n,H)$ can be determined by ensuring
that there is no such local obstruction.

\section{Concluding remarks}

We established a far-reaching generalization of Ore's Theorem,
supplying the precise value of the extremal number $ex(n,H)$ for
a large number of graphs $H$ on $n$ vertices. Note that the result
shows that any graph on $n$ vertices and more than
${{n-1} \choose 2} + \delta-1$ edges is {\em universal} for the
class of all graphs with $n$ vertices, minimum degree at most
$\delta$ and maximum degree at most  $\sqrt n/200$, that is, it
contains all of them.

The extremal graph in Ore's Theorem is unique for all $n > 5$,
that is, when $H=C_n$ is the cycle of  length $n>5$, the only graph
on $n$ vertices with $ex(n,H)={{n-1} \choose 2}+1$ edges containing
no copy of $H$ is $K_n-S_{1,n-2}$, that is,
the graph obtained from $K_n$ by deleting
a star with $n-2$ edges. This is {\em not} 
the case in our more general Theorem
\ref{t:main}. Indeed, let $H$ be, for example,
any bounded degree graph on $n$ vertices 
in which all vertices but one have degree at least $3$, and one
vertex, call it $v$, is of degree $2$ and its two neighbors $x$ and
$y$ in $H$
are adjacent. By Theorem \ref{t:main}, $~ex(n,H)={{n-1} \choose
2}+1$, and one graph on $n$ vertices and $ex(n,H)$ edges
containing no copy of
$H$ is $K_n-S_{1,n-2}$. There is, however, another extremal graph-
the graph $T$ obtained from $K_n$ by deleting a vertex-disjoint union
of a star with $n-3$ edges and a single edge. Indeed, assuming $H$
is a subgraph of such a graph, then the apex of the deleted star
must play the role of $v$, but then its only two neighbors in
$T$ must play the roles of $x$ and $y$, which is impossible, as 
$x$ and $y$ are not adjacent in $T$.

As mentioned in the remark following the proof of
Proposition \ref{p:main}, all our counter-examples to the 
conjecture of \cite{GPW-2012}
regarding the extremal numbers $ex(n,H)$ for hypergraphs $H$
are based on a local obstruction. It seems interesting to decide
if these are all the possible examples. One way to formalize
this question is the following.
\bigskip

\noindent
{\it Question.}\, Is it true that for any $k \geq 2$ and any
$\Delta>0$ there is
an $f=f(\Delta)$ so that for any $k$-uniform hypergraph $H$ on
$n$ vertices with no isolated vertices and with maximum degree
at most $\Delta$, any $k$-uniform hypergraph on $n$ vertices which contains
no copy of $H$ and has the maximum possible number, $ex(n,H)$, of
edges, must contain a complete $k$-hypergraph on at least
$n-f$ vertices?
\bigskip

\noindent
Our proof of Theorem \ref{t:main} implies the validity of
this assertion for graphs (that is, for the case $k=2$).
Indeed, if the maximum degree $d(v_1)$ of the graph $G$
of missing edges considered in the proof satisfies, say,
$n-d(v_1)-1 \geq 6 \delta$, then as shown in the proof,
stage 1 can be completed. In all other stages we can allow
more missing edges and the proof can still be carried out 
with no  change. On the other hand, if $n-d(v_1)-1 < 6 \delta$,
then the graph whose only missing edges are those of $G$ contains
a complete graph on more than $n-12 \delta$ vertices. Therefore,
despite the fact that the extremal graph in the theorem is not
unique, in general, as mentioned above, all extremal examples
must contain a clique on nearly all the vertices, and thus satisfy
the statement in the previous question.

An equivalent formulation of Theorem \ref{t:main} 
is in terms of graph
packing. We have shown that for large $n$,
any two graphs on $n$ vertices,
where one graph has minimum degree $\delta>0$ and maximum
degree at most $\sqrt n/200$ and the other has at most
$n-1-\delta$ edges, pack. There is an extensive literature
dealing with sufficient conditions ensuring that two graphs
$G$ and $H$ on $n$ vertices pack. The main open conjecture on the
subject is the one of Bollob\'as and Eldridge 
\cite{BE-78} asserting that if
$(\Delta(G)+1)(\Delta(H)+1) \leq n+1$ then $G$ and $H$ pack.
Sauer and Spencer (\cite{SS-78}, see also Catlin's paper \cite{Ca-74}) 
proved
that this is the case if 
$2 \Delta(G) \Delta(H) <n$. For a survey of packing results including
many extensions, variants and relevant references, see
\cite{KKY-09}. 
\vspace{0.5cm}

\noindent
{\bf Acknowledgment}\, We would like to thank Roman Glebov
for helpful comments. We also thank Andrew McConvey for pointing out an inaccurate point in an earlier published version of the paper.

\bibliographystyle{plain}

\bibliography{oregen}

\end{document}